\documentclass[11pt]{amsart}

\usepackage[T1]{fontenc}
\usepackage[utf8]{inputenc}

\usepackage[english]{babel}
\usepackage{geometry}
\usepackage{xcolor}
\usepackage{amsmath,amssymb,amsthm,amsfonts,stmaryrd,mathtools}
\usepackage{bbm}
\usepackage[activate={true,nocompatibility},spacing,kerning]{microtype}

\definecolor{darkred}{rgb}{0.5,0,0}
\definecolor{gcolor}{rgb}{0.004,0.396,0.741}	
\usepackage[pdftex,
            colorlinks,
            hyperfootnotes=false,
            pdffitwindow=true,
            plainpages=false,
            pdfpagelabels=true,
            pdfpagemode=UseOutlines,
            pdfpagelayout=TwoPageRight,
            pdftitle={Algebraic Independence of an Airy Function, Its Derivative, and an Antiderivative},
            pdfauthor={Folkmar Bornemann, TU Munich},
            hyperindex,
            setpagesize=false,
            filecolor=purple,
            urlcolor=darkred,
            citecolor=gcolor,
            linkcolor=gcolor]{hyperref}

\theoremstyle{plain}
\newtheorem{theorem}{Theorem}
\newtheorem*{Shidlovskii}{Shidlovskii's lemma}
\newtheorem{lemma}{Lemma}
\newtheorem{corollary}{Corollary}
\theoremstyle{remark}
\newtheorem*{remark}{Remark}

\input{macros.sty}
\begin{document}

\title[Algebraic Independence of Airy Function, Derivative, and Antiderivative]{Algebraic Independence of an Airy Function, Its Derivative, and Antiderivative}
\author{Folkmar Bornemann}
\address{Department of Mathematics, Technical University of Munich, 80290 Munich, Germany}
\email{bornemann@tum.de}


\begin{abstract}
Using tools from the Siegel--Shidlovskii theory of transcendental numbers, we prove that a nontrivial solution of the Airy equation, its derivative, and an antiderivative are algebraically independent over the field of rational functions. Courtesy of Michael Singer, the result is also derived from general considerations in differential Galois theory.
\end{abstract}

\maketitle

We consider the Airy equation in the complex domain, $u'' = z u$.
Any  solution $u\neq 0$ is a transcendental entire function of non-integer order $3/2$ \cite[Prop.~5.1]{MR1207139}, which thus has infinitely many zeros \cite[p.~26]{MR589888}. As a result, the logarithmic derivative $u'/u$ has infinitely many poles, so it cannot be algebraic over the field $\C(z)$ of rational functions. Therefore, by a classical lemma of Siegel\footnote{Siegel's lemma -- \cite[p.~215]{zbMATH02563202}, \cite[p.~60]{MR32684}; see also the expositions in  \cite[Lemma~6.2]{MR1033015}, \cite[Lemma~14.3.1]{MR1207139} -- states that if the differential equation
\[
w'' + a_1 w' + a_0 w = 0,\quad a_0,a_1 \in \C(z),
\]
has a solution $w_0\neq 0$ such that $w_0$, $w_0'$ are algebraically dependent over $\C(z)$, then it has a solution $w_1\neq 0$ whose logarithmic derivative $w_1'/w_1$ is algebraic over $\C(z)$.} on linear second-order differential equations, the functions $u$ and~$u'$ are algebraically independent over $\C(z)$; see also \cite[Thm.~1]{MR683055} and \cite[Prop.~14.4.3]{MR1207139} for a generalization to algebraic independence over a larger field of meromorphic functions.

For a study in random matrix theory \cite{BornemannRMT}, we had to extend the algebraic in\-de\-pen\-dence by adding an antiderivative of $u$. So, the purpose of this note is to prove the following result.

\begin{theorem}\label{thm} Let $u\neq 0$ be a solution of $u''=z u$. Then, the functions $u, u', U$, where $U$ satisfies $U'=u$, are algebraically independent over $\C(z)$.
\end{theorem}
The proof\/\footnote{Another proof, based on the machinery of differential Galois theory, is sketched in the Appendix.} given below is elementary and self-contained, except for the following facts:
\begin{itemize}\itemsep=3pt
\item as already explained, any solution $u\neq 0$ of~the Airy equation and its logarithmic derivative $u'/u$ are transcendental over $\C(z)$;
\item a lemma of Shidlovskii, which has a comparatively short algebraic proof in terms of resultants and which generalizes the first step in the proof of Siegel's lemma. 
\end{itemize}

\begin{Shidlovskii}[\protect{\cite[Corollary p.~188]{MR1033015}}]\hspace*{-2mm}\footnote{Here, we have specialized Shidlovskii's lemma to the homogeneous case with polynomial coefficients.} 
Suppose that the entire functions $w_1,\ldots,w_n$ satisfy a homogeneous system of first-order differential equations with polynomial coefficients,
 \[
\begin{pmatrix}
w_1'\\
\vdots\\
w_n'
\end{pmatrix} = A
\begin{pmatrix}
w_1\\
\vdots\\
w_n
\end{pmatrix},\quad A \in \C[z]^{n \times n}.
\]
Let $w_1,\ldots,w_n$ have transcendence degree $n-1$ over $\C(z)$ so that they are connected by an algebraic equation of the form $p(z,w_1(z),\ldots,w_n(z)) = 0$, where $p \in \C[z,y_1,\ldots,y_n]$ is irreducible. Let $D$ denote the associated differential operator
 \[
 D = \frac{\partial}{\partial z} + \sum_{j,k=1}^n A_{jk}\cdot y_k \frac{\partial}{\partial y_j}.
 \]
 Then, $D p = \omega p$ factorizes as polynomials of $n+1$ independent variables, with $\omega \in \C[z]$. 
\end{Shidlovskii}

\begin{remark}
We note that for any $q\in \C[z,y_1,\ldots,y_n]$ the polynomial $Dq \in \C[z,y_1,\ldots,y_n]$ is constructed to evaluate the total derivative along the field of solutions; namely, it satisfies
\[
\frac{d}{dz} q(z,w_1(z),\ldots w_n(z)) = Dq(z,w_1(z),\ldots,w_n(z)).
\]
\end{remark}

Now, writing $w_1 = u$, $w_2 = u'$ and $w_3 = U$, the Airy equation $u''=zu$ and the defining equation $U'=u$ for an antiderivative induce a system of $n=3$ first-order equations, 
\[
w_1' = w_2,\quad w_2' = z w_1, \quad w_3' = w_1.
\]
Hence, the associated differential operator takes the form 
\[
D = D_{\text{Airy}}+ y_1 \frac{\partial}{\partial y_3}, \quad D_{\text{Airy}} =  \frac{\partial}{\partial z} + y_2 \frac{\partial}{\partial y_1} + z y_1 \frac{\partial}{\partial y_2}.
\]
We need two preparatory lemmas for the proof of Theorem~\ref{thm}.

\begin{lemma}\label{lem:1} Let be $p \in \C[z,y_1,y_2]$. Then $D_{\text{\rm Airy}} p$ is divisible by $p$ if and only if $p$ is constant.
\end{lemma}

\begin{proof}\hspace*{-1.5mm}\footnote{The proof is modeled after the second step in the proof of Siegel's lemma \cite[pp.~61--62]{MR32684}; see also the expositions in \cite[pp.~210--211]{MR1033015} and \cite[pp.~291--292]{MR1207139}. Alternatively, Lemma~\ref{lem:1} can be proved by
inserting the particular Airy function $\Bi(z)$ and its derivative $\Bi'(z)$ as $y_1, y_2$ into \eqref{eq:1}, so that after integration 
\[
p(z,\Bi(z),\Bi'(z)) = \varrho e^{az^2/2 + b z}, \quad \varrho \in \C.
\]
We then conclude $a = b = 0$, $p=\varrho$ from employing, as $z\to\infty$ along the positive real axis, the Poincaré-type expansions \cite[§9.7(ii)]{MR2723248} of the functions $\Bi(z)$ and $\Bi'(z)$. However, we prefer Siegel's elegant algebraic approach.}  
\,Suppose that $D_{\text{\rm Airy}} p$ is divisible by $p$ (the converse direction of the assertion is obvious).
Since the total degree of $D_{\text{\rm Airy}} p$ in the variables $y_1,y_2$ does not exceed the corresponding degree of $p$, and the degree of $D_{\text{\rm Airy}} p$ in $z$ may be $1$ greater than the degree of $p$, it follows that the quotient -- when $D_{\text{\rm Airy}} p$ is divided by $p$ -- is a polynomial in $z$ of degree at most $1$. Hence, as polynomials in three independent variables,
\begin{equation}\label{eq:1}
D_{\text{\rm Airy}}p = (az+b)p,\quad a,b \in \C.
\end{equation}
We aggregate into $q\in \C[z,y_1,y_2]$ all the terms in $p$ that have largest total degree in $y_1,y_2$. So, $q$ is a homogeneous polynomial in~$y_1,y_2$ of some degree $m$ with coefficients in $\C[z]$. Since application of the operator $D_\text{Airy}$ either preserves the degree of a homogeneous term in $y_1, y_2$ or sends that term to zero, \eqref{eq:1} induces 
\[
D_{\text{\rm Airy}} q = (az+b)q.
\]
Now, given a fundamental system $u_1, u_2$ of the Airy equation, we insert the general solution $u = c_1 u_1 + c_2 u_2$ with $c_1, c_2 \in \C$ for $y_1$ 
and its derivative $u'$ for $y_2$. This way, we obtain
\[
\frac{d}{dz} q(z,u(z),u'(z)) = D_{\text{\rm Airy}}q(z,u(z),u'(z)) = (az+b) q(z,u(z),u'(z)).
\]
Integration yields
\[
q(z,u(z),u'(z)) = r(c_1,c_2) e^{a z^2/2 + bz}, \quad r \in \C[c_1,c_2],
\]
where $r$ is homogeneous of degree $m$. 

Suppose that $m=0$. Then, by construction, we have that $q=p$, and the polynomial $p$ depends only on $z$. We also have $r=\varrho$ for some constant $\varrho \in \C$. Since the exponential term represents a polynomial if and only if $a=b=0$, we conclude that $p=\varrho$ is constant.

Finally, contrary to what has to be shown, suppose that $m\geq 1$. In this case, we can determine values of $c_1, c_2$, not both zero, such that $r(c_1,c_2) = 0$. The corresponding particular solution $u = c_1 u_1 + c_2 u_2$ of the Airy equation is nontrivial and satisfies $q(z,u(z),u'(z)) = 0$. If $q$ does not depend on $y_2$, the solution $u$ is algebraic over $\C(z)$, which is impossible. So, $q$ depends on $y_2$ and, by homogeneity, we obtain in a region excluding the zeros of $u$ that
\[
q(z,1,u'(z)/u(z)) = 0.
\]
Therefore, the logarithmic derivative $u'/u$ is algebraic over $\C(z)$, which is impossible, too.
\end{proof}

\begin{corollary}\label{cor:1} Let $u\neq 0$ solve $u''=z u$. Then, $u, u'$ are algebraically independent over $\C(z)$.
\end{corollary}

\begin{proof} Suppose to the contrary that $w_1=u, w_2=u'$ are algebraically dependent. 
Then, since $w_1$ is transcendental, the transcendence degree of $w_1, w_2$ is $1$. By Shidlovskii's lemma, there is an irreducible polynomial $p \in \C[z,y_1,y_2]$, such that $p(z,w_1(z),w_2(z)) = 0$ with
\[
D_\text{Airy}p = \omega p,\quad \omega \in \C[z].
\]
By Lemma~\ref{lem:1}, $p$ must be constant, contradicting the irreducibility.
\end{proof}

\begin{lemma}\label{lem:2} Let be $p \in \C[z,y_1,y_2]$. Then $D_{\text{\rm Airy}} p + cy_1 = 0$ for some $c \in \C$ if and only if $c=0$ and $p$ is constant.
\end{lemma}

\begin{proof} Suppose that $D_\text{Airy}p + c y_1=0$ for some $c \in \C$ (the converse direction of the assertion is obvious).
Expanding in powers of $y_1$, we write
\[
p(z,y_1,y_2) = \sum_{k=0}^m q_k(z,y_2) y_1^k,\quad q_k \in \C[z,y_2].
\]
Applying the operator $D_\text{Airy}$ and setting $q_{m+2}=q_{m+1}= q_{-1} = 0$, we get
\[
D_{\text{\rm Airy}} p = \sum_{k=0}^{m+1} \left(\frac{\partial q_k}{\partial z} + z \frac{\partial q_{k-1}}{\partial y_2} + (k+1) y_2 q_{k+1}\right) y_1^k.
\]
Now, in $D_\text{Airy} p + c y_1 = 0$, comparing the coefficients of the powers $y_1^k$ for $k=0,1$ gives
\begin{equation}\label{eq:2}
\frac{\partial q_0}{\partial z} + y_2 q_1 =0,\quad \frac{\partial q_1}{\partial z} + z \frac{\partial q_0}{\partial y_2} + 2 y_2 q_2 + c = 0.
\end{equation}
If we develop the polynomials $q_0, q_1, q_2$ in powers of $y_2$ with coefficients in $\C[z]$, writing
\[
q_0 = f + g y_2 + O(y_2^2),\quad q_1 = h + O(y_2),\quad q_2 = O(1), \quad f, g, h \in \C[z],
\]
we obtain from comparing in \eqref{eq:2} the coefficients of the powers $y_2^k$ for $k=0,1$ that
\[
f' = 0,\quad g' + h = 0,\quad h' + z g + c = 0.
\]
Hence, $g$ is a polynomial solution of the (inhomogeneous) Airy equation $g'' = z g + c$. This is impossible for $g\neq 0$ since then $\deg (g'') < \deg (z g + c)$. As a result, we get $g= 0$ and a~fortiori also $c=0$, which implies $D_\text{Airy} p = 0$ so that $p$ is constant by Lemma~\ref{lem:1}.
\end{proof}

\begin{proof}[Proof of Theorem~\ref{thm}]\hspace*{-1.75mm}\footnote{The proof is modeled after the one given by Shidlovskii \cite[Lemma 6, p.~190]{MR1033015} that establishes, for fixed $\alpha \in \C\setminus\{0\}$ and $\lambda\in \C\setminus\Z$, the algebraic independence of the entire functions $e^{\alpha z}$ and $_1F_1(1,\lambda;z)$ 
over $\C(z)$.} 
\, Suppose to the contrary that $w_1=u, w_2=u', w_3=U$ are algebraically dependent. 
Then, because of the algebraic independence of $w_1,w_2$ established in Corollary~\ref{cor:1}, the transcendence degree of $w_1, w_2, w_3$ is $2$. By Shidlovskii's lemma, there is an irreducible polynomial $p \in \C[z,y_1,y_2,y_3]$, such that $p(z,w_1(z),w_2(z),w_3(z)) = 0$ with
\[
Dp = \omega p,\quad \omega \in \C[z].
\]
We write 
\[
p(z,y_1,y_2,y_3) = \sum_{k=0}^m q_k(z,y_1,y_2) y_3^k,\quad q_k \in \C[z,y_1,y_2],
\]
with $q_m \neq 0$, where $m\geq 1$ since $w_1$ and $w_2$ are algebraically independent. Applying the operator $D$ and setting $q_{m+1}=0$, we get
\[
D p = \sum_{k=0}^{m} \left(D_\text{Airy} q_k  + (k+1) y_1 q_{k+1}\right)y_3^k.
\]
Now, in $Dp=\omega p$, we start with comparing the coefficients of $y_3^m$ which gives $D_\text{Airy} q_m = \omega q_m$.
Hence, Lemma~\ref{lem:1} implies that $\omega = 0$ and $q_m = \varrho$ for some constant $\varrho \in \C$. Next, comparing the coefficients of $y_3^{m-1}$ yields 
\[
D_\text{Airy} q_{m-1} + m\varrho  y_1 = 0,
\]
which by Lemma~\ref{lem:2} implies $\varrho = 0$, contradicting $q_m \neq 0$. Thus, $w_1,w_2,w_3$ are algebraically independent, which proves Theorem~\ref{thm}.
\end{proof}

We conclude with two brief remarks. First, Corollary~\ref{cor:1} is also true for solutions $u\neq 0$ of Painlevé~II (with a zero parameter), that is, the nonlinear differential equation
\[
u'' = z u + 2 u^3;
\] 
see \cite[Theorem~21.1]{MR1960811}, \cite[Theorem~9.2]{MR3558513}. We conjecture that Theorem~\ref{thm} holds true for Painlevé~II as well. On the other hand -- as for  Painlevé II; see \cite[Eq.~(1(ii))]{MR3558513} -- a solution $u\neq 0$ of the Airy equation, its derivative $u'$, and an antiderivative $U_2$ of $u^2$ are algebraically dependent over $\C(z)$:
\[
u'^2 = z u^2 - U_2 + c, \quad U_2' = u^2,
\]
where $c\in \C$ is a constant of integration.

Second, though we have used tools from transcendental number theory, Theorem~\ref{thm} has no immediate arithmetic consequences (such as the algebraic or linear independence over $\Q$ for certain function values at non-zero algebraic arguments). The standard fundamental system of the Airy equation has power series with rational coefficients, namely (cf. \cite[§9.4]{MR2723248})
\begin{align*}
u_1(z) &= \,_0F_1\big(\tfrac23;\tfrac{z^3}{9}\big) = 1 + \frac{1}{3!}z^3 + \frac{1\cdot 4}{6!} z^6 + \frac{1\cdot 4\cdot 7}{9!}z^9 + \cdots,\\*[2mm]
u_2(z) &= z\cdot\/ _0F_1\big(\tfrac43;\tfrac{z^3}{9}\big) =  z + \frac{2}{4!}z^4 + \frac{2\cdot 5}{7!} z^7 + \frac{2\cdot 5\cdot 8}{10!}z^{10}+ \cdots,
\end{align*}
but neither of these solutions (nor any rational linear combination of them) is a Siegel $E$-function (\cite[p.~223]{zbMATH02563202}, \cite[p.~79]{MR1033015}) or $G$-function (\cite[p.~239]{zbMATH02563202}, \cite[p.~438]{MR1033015}): 
first, writing the non-zero coefficients of $z^n$ as $a_n/n!$ with integer $a_n$, the $a_n$ have no growth bound better than $$a_n = O(e^{-n/3} n^{n/3-1/6}),$$ which rules out the class of $E$-functions; second, writing the  non-zero coefficients as $1/a_n'$ with integer $a_n'$, the $a_n'$ have no  growth bound better than $$a_n' = O(e^{-2n/3} n^{2n/3+2/3}),$$ which rules out the class of $G$-functions.

\appendix
\renewcommand{\sectionname}{}
\renewcommand{\thesection}{Appendix}
\section{Proof of Theorem~\ref{thm} based on differential Galois theory}\label{app:airykernel} 
~\\*[-12mm]
\begin{center}
{\footnotesize -- courtesy of Michael Singer, Department of Mathematics, North Carolina State University --}
\end{center}

\subsubsection*{General considerations}
Let $k = \C(z)$ be the field of rational functions and let $L(u)=0$ be a linear differential equation of order $n$ with coefficients in $k$. Let $u_1,\ldots,u_n$ be a fundamental system of meromorphic solutions so that
\[
K = k\left(u_1,\ldots,u_n,u_1',\ldots,u_n',\ldots,u_1^{(n-1)},\ldots,u_n^{(n-1)}\right)
\]
is the Picard--Vessiot field extension associated with the differential equation. Two important facts about the differential Galois group $G = \text{Gal}(K|k)$ are (cf., e.g., \cite[§6.2]{MR2778109}, \cite[§1.4]{MR1960772}):\footnote{Up to isomorphism over $k$, $G$ does not dependent on the choice of the fundamental system.}
\begin{itemize}\itemsep=3pt
\item $G$, considered as a subgroup of $\GL_n(\C)$, is an algebraic group;
\item the dimension $\dim_\C G$ is the transcendence degree of $K$ over $k$.
\end{itemize}
The following theorem characterizes the transcendence of antiderivatives (writing $\partial = d/dz$).

\begin{theorem}\label{thm:2} Let the differential operator $L \in k[\partial]$ of order $n$ be irreducible and let $K$ be the associated Picard--Vessiot extension. If $u\neq 0$ satisfies $L(u)=0$, then an antiderivative $U$ of $u$ is algebraic over $K$ if and only if the inhomogeneous adjoint equation $L^*(v)=1$ has a solution $v\in k$. In this case, $U$ and $u, u',\ldots,u^{(n-1)}$ are connected by the algebraic equation
\[
U=c-\pi(u,v), \quad c \in \C,
\]
where $\pi(u,v)$ is the bilinear concomitant of $L$, satisfying the Lagrange identity\/\footnote{Cf., e.g., \cite[§5.3]{MR10757}. The concomitant acts as a differential operator of order $n-1$ on each of its arguments.} 
\[
v L(u) - u L^*(v) = \partial\pi(u,v).
\]
\end{theorem}

For $U\in K$, this theorem rephrases a result that first appeared in \cite[Lemma~4]{MR1094849}, stated there more general in terms of $D$-modules and cohomology groups. If $U$ is algebraic over $K$, it easily follows from using traces that, in fact, $U \in K$. 

\subsubsection*{Application to the Airy Equation}

Let $u, \tilde u$ be a fundamental system of the Airy equation. It is shown in many places (e.g., \cite[p.~44]{MR460303}, \cite[Example~8.15]{MR1960772}, \cite[Example~7.3.3]{MR2778109}, \cite[Theorem~8.1]{MR2800334}) that the differential Galois group of the associated Picard--Vessiot extension $K$ is 
\[
\text{Gal}(K|k) = \SL_2(\C),
\] 
which has dimension~$3$.
Hence, the transcendence degree of $K$ over $k$ is $3$, and any three of the functions $u,u',\tilde u,\tilde u'$  are algebraically independent over $k$,\footnote{Each of them is a rational expression of the other three since the Wronski determinant is a constant.} which implies Corollary~\ref{cor:1}.

The selfadjoint Airy operator $L=\partial^2 - z$ is irreducible: if it factored as $L=L_1\circ (\partial -w)$, then $w\in k$ would be a rational solution of the associated Riccati equation \cite[p.~104]{MR1960772},
\[
w' + w^2 = z.
\] 
An expansion at $z=\infty$ shows this is impossible. Similarly, an expansion at a putative pole shows that the inhomogeneous Airy equation $v''=zv+1$ also has no rational solutions. Hence, Theorem~\ref{thm:2} implies that any antiderivative $U$ of $u$ is transcendental over $K$. In particular, the functions $u, u', U$ are algebraically independent over $k$, 
which proves Theorem~\ref{thm}.

{\small
\subsection*{Acknowledgements} I am grateful to Michael Singer for generously providing the material presented in the Appendix. I thank Alexandre Eremenko for his encouragement and Norbert Steinmetz for pointing out the algebraic dependence of the functions $u$, $u'$, and $U_2$.}

\bibliographystyle{spmpsci}
\bibliography{paper}

\begin{thebibliography}{10}
\providecommand{\url}[1]{{#1}}
\providecommand{\urlprefix}{URL }
\expandafter\ifx\csname urlstyle\endcsname\relax
  \providecommand{\doi}[1]{DOI~\discretionary{}{}{}#1}\else
  \providecommand{\doi}{DOI~\discretionary{}{}{}\begingroup
  \urlstyle{rm}\Url}\fi

\bibitem{MR683055}
Bank, S.B.: On certain properties of solutions of second-order linear
  differential equations.
\newblock Monatsh. Math. \textbf{94}(3), 179--200 (1982)

\bibitem{MR1094849}
Bertrand, D.: Extensions de {$D$}-modules et groupes de {G}alois
  diff\'erentiels.
\newblock In: {$p$}-adic {A}nalysis ({T}rento, 1989), pp. 125--141. Springer,
  Berlin (1990)

\bibitem{BornemannRMT}
Bornemann, F.: Asymptotic expansions of {G}aussian and {Laguerre} ensembles at
  the soft edge {II}: Level densities (2025).
\newblock \urlprefix\url{https://arxiv.org/abs/2503.12644}

\bibitem{MR2778109}
Crespo, T., Hajto, Z.: Algebraic Groups and Differential {G}alois Theory.
\newblock AMS, Providence, RI (2011)

\bibitem{MR1960811}
Gromak, V.I., Laine, I., Shimomura, S.: Painlev\'{e} Differential Equations in
  the Complex Plane.
\newblock Walter de Gruyter, Berlin (2002)

\bibitem{MR2800334}
Hubbard, J.H., Lundell, B.E.: A first look at differential algebra.
\newblock Amer. Math. Monthly \textbf{118}(3), 245--261 (2011)

\bibitem{MR10757}
Ince, E.L.: Ordinary {D}ifferential {E}quations.
\newblock Dover Publications, New York (1944)

\bibitem{MR460303}
Kaplansky, I.: An Introduction to Differential Algebra, 2nd edn.
\newblock Hermann, Paris (1976)

\bibitem{MR1207139}
Laine, I.: {Nevanlinna Theory and Complex Differential Equations}.
\newblock Walter de Gruyter, Berlin (1993)

\bibitem{MR589888}
Levin, B.J.: Distribution of Zeros of Entire Functions.
\newblock AMS, Providence, RI (1980)

\bibitem{MR2723248}
Olver, F.W.J., Lozier, D.W., Boisvert, R.F., Clark, C.W. (eds.): N{IST}
  {Handbook of Mathematical Functions}.
\newblock Cambridge Univ. Press, Cambridge (2010)

\bibitem{MR1960772}
van~der Put, M., Singer, M.F.: Galois Theory of Linear Differential Equations.
\newblock Springer, Berlin (2003)

\bibitem{MR1033015}
Shidlovskii, A.B.: Transcendental Numbers.
\newblock Walter de Gruyter, Berlin (1989)

\bibitem{zbMATH02563202}
Siegel, C.L.: {\"U}ber einige {Anwendungen} diophantischer {Approximationen}.
\newblock Abh. Preu{{\ss}}. Akad. Wiss., Phys.-Math. Kl. \textbf{1929},
  209--266 (1929).
\newblock Reprinted with translation and commentary in: Zannier, U. (ed.), On
  some applications of {D}iophantine approximations, Edizioni della Normale,
  Pisa (2014)

\bibitem{MR32684}
Siegel, C.L.: Transcendental {N}umbers.
\newblock Princeton University Press, Princeton, NJ (1949)

\bibitem{MR3558513}
Steinmetz, N.: A unified approach to the {P}ainlev\'e{} transcendents.
\newblock Ann. Acad. Sci. Fenn. Math. \textbf{42}(1), 17--49 (2017)

\end{thebibliography}

\end{document}